\documentclass[a4paper,12pt]{article}

\usepackage{color,makeidx,bm,latexsym,amscd,stmaryrd,psfrag,amsmath,amssymb,enumerate,amsthm,float,graphicx,geometry}
\usepackage{caption}
\usepackage{subcaption}
\usepackage[all]{xy}

\usepackage{tikz,epic,eso-pic}

\theoremstyle{plain}
\newtheorem{theorem}{Theorem}[section]
\newtheorem{corollary}[theorem]{Corollary}
\newtheorem{lemma}[theorem]{Lemma}

\newtheorem{problem}[theorem]{Problem}

\theoremstyle{definition}
\newtheorem{definition}[theorem]{Definition}
\newtheorem{remark}[theorem]{Remark}
\newtheorem{example}[theorem]{Example}
\newtheorem{proposition}[theorem]{Proposition}

\newcommand{\N}{\mathbb{N}}
\newcommand{\Q}{\mathbb{Q}}
\newcommand{\R}{\mathbb{R}}

\newcommand{\Z}{\mathbb{Z}}
\newcommand{\zono}{\mathcal{Z}}
\newcommand{\vol}{\operatorname{vol}}
\newcommand{\relvol}{\operatorname{relvol}}
\newcommand{\aff}{\operatorname{aff}}
\newcommand{\den}{\operatorname{den}}
\newcommand{\conv}{\operatorname{conv}}
\newcommand{\lcm}{\operatorname{lcm}}
\newcommand{\GCD}{\operatorname{GCD}}
\newcommand{\LCM}{\operatorname{LCM}}
\newcommand{\Ker}{\operatorname{Ker}}

\newcommand{\op}{\operatorname{op}}
\newcommand{\interior}{\operatorname{int}}
\newcommand{\rad}{\operatorname{rad}}

\title{Ehrhart quasi-polynomials of almost integral polytopes} 
\author{Christopher de Vries\thanks{Department of Mathematics, University of Bremen, Bibliothekstra\ss e 1, 28359 Bremen, GERMANY/Hokkaido University
E-mail: cvries@uni-bremen.de }, 
Masahiko Yoshinaga\thanks{
Department of Mathematics, 
Graduate School of Science, 
Osaka University,
Toyonaka, Osaka 560-0043, JAPAN 
E-mail: yoshinaga@math.sci.osaka-u.ac.jp}}
\date{\today}
\begin{document}
\maketitle

\begin{abstract} 
A lattice polytope translated by a rational vector is called 
an almost integral polytope. 
In this paper, we study Ehrhart quasi-polynomials of 
almost integral polytopes. 
We study the connection between the shape of polytopes 
and the algebraic properties of the Ehrhart quasi-polynomials. 
In particular, we prove that lattice zonotopes and centrally symmetric 
lattice polytopes are characterized by Ehrhart quasi-polynomials of 
their rational translations. 
\end{abstract}


\section{Introduction}
\label{sec:intro}

A \emph{polytope} $P$ is the convex hull of finitely many points 
in $\R^d$. 
A polytope $P$ is called a lattice polytope (resp. rational polytope) if 
all its vertices are contained in $\Z^d$ (resp. $\Q^d$). 
The set of lattice points in a rational polytope $P$ is a 
fundamental object of study in 
enumerative combinatorics (\cite{bec-rob, st-ec1}). 
Notably, it is known that the function 
$\Z_{>0}\ni t\longmapsto L_P(t):=\#(t P\cap\Z^d)$ is 
a quasi-polynomial \cite{ehrhart}. 
In other words, there exists a positive 
integer $\rho>0$ and polynomials 
$f_1(x), f_2(x), \dots, f_\rho(x)\in\Z[x]$ such that 
\begin{equation}
L_P(t)=
\begin{cases}
f_1(t),  & \mbox{ if }t\equiv 1\mod\rho,\\
f_2(t),  & \mbox{ if }t\equiv 2\mod\rho,\\
 & \vdots\\
f_\rho(t),  & \mbox{ if }t\equiv \rho\mod\rho. 
\end{cases}
\end{equation}
Here, 
$\rho$ is referred to as the \emph{period}, and 
$f_1(x), \dots, f_\rho(x)$ are called the \emph{constituents}. 
(We identify the $\rho$-th constituent $f_\rho(x)$ with 
the $0$-th one $f_0(x)$.) 
$L_P(t)$ is called the \emph{Ehrhart quasi-polynomial} of $P$.

It is obvious that a multiple of a period of $L_P(t)$ is again a period 
of $L_P(t)$. We define 
the \emph{minimal period} as the smallest possible period, 
and therefore, every period is a multiple of this minimal period. 
It is a well-known fact that the minimal 
period divides the least common multiple ($\LCM$) of the 
denominators of the coordinates of the vertices of $P$. 
For the sake of simplicity, we will primarily focus on 
the minimal period. 
However, it is important to note that 
this restriction is not essential for our purposes 
(See Proposition \ref{prop:indgcd} and Remark \ref{rem:indsym}.)

Generally, the constituents of $L_P(t)$ are usually distinct from 
each other. 
However, in numerous examples, 
some of the constituents may be found to be 
identical. As a result, the number of distinct constituents becomes 
strictly less than 
the minimal period $\rho$. Typical examples are as follows.

\begin{example}
\label{ex:01}
Consider the following polytopes: 
\begin{itemize}
\item
$P_1=\frac{1}{9}\cdot [0, 1]^3$, which is a 
$3$-cube with side length $\frac{1}{9}$, 
\item 
$P_2=(\frac{5}{9}, \frac{5}{9}, \frac{2}{3})^t+
\operatorname{Conv}\{\pm e_i\mid i=1, 2, 3\}$, 
an octahedron translated by a rational vector 
(where $e_1, e_2, e_3$ are the standard basis vectors of $\R^3$), 
\item 
$P_3=(\frac{1}{9}, \frac{2}{9}, \frac{1}{3})^t+[0, 1]^3$, 
a unit cube translated by a rational vector. 
\end{itemize}
The Ehrhart quasi-polynomials 
$L_{P_1}(t)$, $L_{P_2}(t)$, and $L_{P_3}(t)$ 
all share the same minimal period $\rho=9$. 
The constituents of 
$L_{P_1}(t)$ are $f_k(t)=\left(\frac{t+9-k}{9}\right)^3$ for 
$k=0, 1, \dots, 8$, they are mutually distinct. 
In contrast, the constituents of 
$L_{P_2}(t)$ and $L_{P_3}(t)$ are as follows. 
\[
\begin{split}
L_{P_2}(t)&=
\begin{cases}
\frac{4}{3}t^3-\frac{4}{3}t,  & (t\equiv 1, 8\mod 9),\\
\frac{4}{3}t^3+\frac{2}{3}t,  & (t\equiv 2, 7\mod 9),\\
\frac{4}{3}t^3+t^2+\frac{2}{3}t,  & (t\equiv 3, 6\mod 9),\\
\frac{4}{3}t^3-\frac{1}{3}t,  & (t\equiv 4, 5\mod 9),\\
\frac{4}{3}t^3+2t^2+\frac{8}{3}t+1,  & (t\equiv 9\mod 9),
\end{cases}
\\
&
\\
L_{P_3}(t)&=
\begin{cases}
t^3  & (t\equiv 1, 2, 4, 5, 7, 8\mod 9),\\
t^3+t^2  & (t\equiv 3, 6\mod 9),\\
(t+1)^3  & (t\equiv 9\mod 9). 
\end{cases}
\end{split}
\]
Note that in $L_{P_2}(t)$ and $L_{P_3}(t)$ some of constituents coincide. 
\end{example}

The motivation for our paper is to investigate 
the connection between coincidences of constituents and 
the shape of the polytope $P$. 
To formalize these ``coincidences among constituents'', 
we introduce the following notion. 

\begin{definition}
Let $L(t)$ be a quasi-polynomial with period $\rho$ and 
constituents $f_1, \dots, f_{\rho-1}, f_\rho (=f_0)$. 
\begin{itemize}
\item[(1)] 
We say that $L(t)$ is \emph{symmetric} if $f_k=f_{\rho-k}$ for $0\leq k\leq\rho$. 
\item[(2)] 
We say that 
$L(t)$ has the \emph{$\GCD$-property} if 
$f_k=f_\ell$ whenever $\GCD(\rho, k)=\GCD(\rho, \ell)$. 
\end{itemize}
\end{definition}
Clearly, if a quasi-polynomial satisfies the $\GCD$-property, 
then it is symmetric. 

\begin{remark}
Quasi-polynomials with the $\GCD$-property naturally appear 
in the theory of hyperplane arrangements. 
See \S \ref{sec:hyp} further details. 
\end{remark}

In Example \ref{ex:01}, 
$L_{P_2}(t)$ is symmetric and 
$L_{P_3}(t)$ satisfies the $\GCD$-property. 
Surprisingly, these properties of Ehrhart quasi-polynomials 
are closely related to the fact that 
``$P_2$ is centrally symmetric'' and 
``$P_3$ is a zonotope'', respectively. 
In fact, the primary objective of this paper is to establish 
the correspondence between two columns of the following table. 
\begin{equation}
\begin{array}{c|c}
\mbox{Shape of $P$}&\mbox{Property of $L_P(t)$}\\
\hline
\mbox{General polytope}&\mbox{General quasi-polynomial}\\
\cup & \cup\\
\mbox{Centrally symmetric}&\mbox{Symmetric}\\
\cup & \cup\\
\mbox{Zonotope}&\mbox{$\GCD$-property}
\end{array}
\end{equation}
The main results (Theorem \ref{charsym} and Theorem \ref{charzono}) 
establish that for \emph{almost 
integral polytopes} (rationally translated lattice polytopes), 
the properties of the left column imply those of the right. 
Furthermore, 
we can also characterize the left column using the properties of 
the right column.

The paper is organized as follows: 

In \S \ref{sec:eqp}, after introducing basic notions, we recall 
a formula by Ardila-Beck-McWhirter, 
which expresses the Ehrhart quasi-polynomial of 
an almost integral zonotope. 
Applying this formula, we demonstrate that the 
Ehrhart quasi-polynomial satisfies the $\GCD$-property. 

In \S \ref{sec:tlpe}, we introduce the concept of the 
\emph{translated lattice point enumerator}, denoted as 
\begin{equation}
L_{(P, c)}(t):=\#((c+t P)\cap\Z^d), 
\end{equation}
where $P$ is a lattice polytope in $\R^d$ and $c\in\R^d$. 
We prove that $L_{(P, c)}(t)$ is a polynomial in $t\in\Z_{>0}$ 
(Theorem \ref{tlpe}). 
Furthermore, we show that the constituents of an almost integral 
polytope can be described in terms of $L_{(P, c)}(t)$ 
(Corollary \ref{cor:consti}). This result readily implies 
that if $P$ is a centrally symmetric almost integral polytope, 
then $L_P(t)$ is symmetric (Corollary \ref{sym1}). 

In \S \ref{sec:symm}, we present the first main result. 
Specifically, we prove that a lattice polytope 
$P$ is centrally symmetric if and only if 
the Ehrhart quasi-polynomial $L_{c+P}(t)$ is symmetric for 
every rational vector $c$ (Theorem \ref{charsym}). 
The ``only if'' part has already been established in 
\S \ref{sec:tlpe}. 
To complete the proof, we show that if $P$ is not 
centrally symmetric, then there exists a rational vector $c$ such that 
$L_{(P, c)}(t)\neq L_{(P, -c)}(t)$. For this, we employ 
Minkowski's result, which characterizes a polytope based on 
normal vectors and volumes of facets. 

In \S \ref{sec:zono}, we present the second main result. 
Specifically, we prove that a lattice polytope $P$ is a 
zonotope if and only if the Ehrhart quasi-polynomial 
$L_{c+P}(t)$ satisfies the $\GCD$-property for every rational 
vector $c$ (Theorem \ref{charzono}). 
The ``only if'' part has already been established in 
\S \ref{sec:eqp}. To complete the proof, we employ an 
involved argument using McMullen's characterization 
of zonotopes in terms of central symmetricity of faces. 
We establish that if $P$ is not a zonotope, 
then there exists a rational vector $c$ with odd denominators 
such that $L_{(P, c)}(t)\neq L_{(P, 2c)}(t)$, which 
implies $L_{c+P}(t)$ does not satisfy the $\GCD$-property. 

In \S \ref{discuss}, we will examine related problems. 
First, we discuss the minimal periods of almost integral 
polytopes and explore the relationship between the 
Ehrhart quasi-polynomials of 
almost integral zonotopes and the characteristic quasi-polynomials of 
hyperplane arrangements. We pose several related questions. 

We conclude this section with an elementary proposition 
concerning the periods of quasi-polynomials. 
This proposition justifies our convention that 
we may always assume a period is minimal. 

\begin{proposition}
\label{prop:indgcd}
Let $Q$ be a quasi-polynomial with the minimal period $\rho_0$ and 
the $i$-th constituent $f_i$. Let $k\in\Z_{>0}$. Then $Q$ has the 
$\GCD$-property for $\rho_0$ if and only if $Q$ has the 
$\GCD$-property for $k\rho_0$. 
\end{proposition}

\begin{proof}
Suppose that $Q$ has the $\GCD$-property for the 
minimal period $\rho_0$. 
Note that $\GCD(\rho_0, i)=\GCD(\rho_0, \GCD(k\rho_0, i))$. 
Hence if 
$\GCD(k\rho_0, i)=\GCD(k\rho_0, j)$, then 
$\GCD(\rho_0, i)=\GCD(\rho_0, j)$. 
Therefore, it satisfies the $\GCD$-property for $k\rho_0$. 

Conversely, suppose $Q$ has the $\GCD$-property for $k\rho_0$. 
Suppose $\GCD(i, \rho_0)=\GCD(j, \rho_0)=d$. We shall prove 
$f_i=f_d=f_j$. Without loss of generality, we may assume that 
$i=d=\GCD(j, \rho_0)$. It suffices to show that there exists 
an integer $m\in\Z$ such that 
$i=\GCD(i, k\rho_0)=\GCD(j+m\rho_0, k\rho_0)$. 
This means, by the $\GCD$-property for $k\rho_0$, that 
$f_i=f_{j+m\rho_0}=f_j$. 
Define 
\[m=\frac{\rad (k)}{\rad (\GCD(\frac{j}{i}, k))}, \] 
where $\rad (a)$ is the product of all primes dividing $a$. 

By using that definition, it follows from 
$\GCD(\frac{j}{i}+m\frac{\rho_0}{i}, \frac{\rho_0}{i})= 
\GCD(\frac{j}{i}, \frac{\rho_0}{i})=1$ that 
$GCD(\frac{j}{i}+m\frac{\rho_0}{i}, k\frac{\rho_0}{i})=
\GCD(\frac{j}{i}+m\frac{\rho_0}{i}, k)$. 
Now let $r\in \mathbb{N}$ be a prime such that $r|k$. 
By the definition of $m$, 
$r$ divides $m$ if and only if $r$ does not divide $\frac{j}{i}$. 
Hence, $\GCD(\frac{j}{i}+m\frac{\rho_0}{i}, k)=1$, and we also have 
$\GCD(\frac{j}{i}+m\frac{\rho_0}{i}, k\frac{\rho_0}{i})=1$. 
Therefore, $\GCD(j+m\rho_0, k\rho_0)=i$. 
\end{proof}

\begin{remark}
\label{rem:indsym}
Similarly, being symmetric is independent of the period.
\end{remark}

\section{Ehrhart quasi-polynomials for rational polytopes}
\label{sec:eqp}

\subsection{Notation}

A polytope $P\subset\R^d$ is called \emph{centrally symmetric} if 
there exists $c\in\R^d$ such that $P=c+(-P)$. 
The point $\frac{c}{2}\in P$ serves as the center of $P$. 
A \emph{zonotope} $\zono(u_1, \dots, u_n)$, 
formed by vectors $u_1,\ldots , u_n\in \mathbb{R}^d$, is 
the Minkowski sum of the line segments $[0, u_i]$. 
In other words,  
\[
\zono(u_1, \dots, u_n)=
\{\lambda_1 u_1+\dots+\lambda_n u_n \mid 
0\leq \lambda_i\leq 1, i=1, \dots, n\}. 
\]
It is easily seen that zonotopes are centrally symmetric. 

For a polytope $P\subset\R^d$, we denote the minimal 
affine subspace containing $P$ by $\aff(P)$. 
Additionally, we denote by $\aff_0(P)$ the linear subspace 
of $\R^d$ parallel to $\aff(P)$ and contains the origin. 
The dimension of a polytope $P$ is defined as 
$\dim \aff_0(P)$ and is denoted by $\dim P$.

Let $P$ be a lattice polytope. 
If $X\subset P$ is an $m$-dimensional face, 
then $\aff(X)\cap\Z^d\simeq\Z^m$. 
The \emph{relative volume} of the polytope $X$ is the volume of $X$ 
rescaled so that the unit cube in $\aff(X)\cap \Z^d\simeq\Z^m$ 
has volume 1. 
We denote the relative volume of $X$ by $\relvol(X)$. 
The $k$-dimensional Euclidean volume is denoted by $\vol_k$. 
Note that if $P\subset \R^d$ is a $d$-polytope 
(a polytope of dimension $d$), then $\relvol(P)=\vol_d(P)$. 
For a rational polytope $P$ of dimension $m\le d$, 
the leading coefficient of 
every constituent (the coefficient in degree $m$)  of $L_P(t)$ 
is the relative volume $\relvol(P)$ of $P$.

\begin{definition}
A polytope $P\subset \mathbb{R}^d$ is called 
\textit{almost integral}, if there exists a lattice  polytope 
$P'\subset \mathbb{R}^d$ and a translation vector 
$c\in \mathbb{Q}^d$, such that $P=c+P'$.
\end{definition}

A period of the Ehrhart quasi-polynomial $L_P$ of an 
almost integral polytope $P=c+P'$ with the translation vector 
$c=(c_1, \dots, c_d)\in \Q^d$ is 
$\den(c):=\lcm\{\den(c_i)|i=1,\ldots,d\}\}$, 
where $\den(c_i)$ denotes the denominator 
of the reduced fraction of $c_i$. 
It is expected that $\den(c)$ is the minimal period of $L_P$. See 
\S \ref{sec:min} for a related discussion.

\subsection{Ehrhart quasi-polynomials for almost integral zonotopes}

The following proposition by Ardila, Beck, and McWhirter 
describes the Ehrhart quasi-polynomial of almost integral zonotopes.

\begin{proposition}\cite[Proposition 3.1]{ard}
\label{aiz}
Let $U\in \Z^d$ be a finite set of integer vectors, and 
let $c\in \Q^d$ be a rational vector. 
Then the Ehrhart quasi-polynomial of 
the almost integral zonotope $c+\zono (U)$ is given by 
\begin{equation}
L_{c+\zono (U)}(t)=\sum_{\substack{W \subseteq U \\ W \textrm{ lin. indep.}}} \chi_W(t)\cdot \relvol(\zono (W)) \cdot t^{|W|}, 
\end{equation}
where 
\[\chi_W(t)=\begin{cases} 1, & \text{if } (tc+\aff (W))\cap \Z^d \neq \varnothing \\ 
0, & \text{otherwise} .\end{cases}\]
\end{proposition}

Let $\rho_0=\den(c)$. 
Then $\rho_0$ is a period of $L_{c+\zono(U)}(t)$. 
Proposition \ref{aiz} says 
that the $k$-th constituent of $L_{c+\zono(U)}(t)$ is given by 
\begin{equation}
\label{k-const}
f_k(t)=
\sum_{\substack{W \subseteq U \\ W \textrm{ lin. indep.}}} \chi_W(k)\cdot \relvol(\zono (W)) \cdot t^{|W|}. 
\end{equation}
The first result of this paper is the following. 
\begin{theorem}\label{zono1}
Let $P=c+\zono(U)\subset \R^d$ be an almost integral zonotope. 
Then $L_P$ 
satisfies the $\GCD$-property. 
\end{theorem}

\begin{proof}
In order to prove the $\GCD$-property for 
the Ehrhart quasi-polynomials of almost integral zonotopes, 
it is sufficient to show that the function $\chi_W(t)$ 
in Proposition \ref{aiz} satisfies the $\GCD$-property. 

Let $W=\{u_1,\ldots, u_k\}\subseteq U$ be linearly independent subset. 
Denote by 
\[\langle W \rangle=\left(\sum_{i=1}^k \R u_i\right) \cap \Z^d,\] 
the intersection of the linear subspace generated by $W$ with $\Z^d$. 
By extending a $\Z$-basis of $\langle W\rangle$ to that of $\Z^d$, 
we have 
$u_{k+1}, \dots, u_d\in\Z^d$ such  that $\Z^d=\langle W\rangle\oplus
\bigoplus_{i=k+1}^d\Z u_i$. 
Decompose $c=(c_1, \dots, c_d)\in\Q^d$ as 
$c=c'+a_{k+1}u_{k+1}+\dots+a_d u_d$, where 
$c'\in \aff_0(W)\cap \Q^d $ and 
$a_i\in\Q$. Then $(tc+\aff (W))\cap \Z^d \neq \varnothing$ 
if and only if 
$ta_{k+1}, \dots, ta_d\in\Z$. 
This is also equivalent to $t$ being divisible by 
$\lcm(\den(a_{k+1}), \dots, \den(a_d))$. 
Note that $\lcm(\den(a_{k+1}), \dots, \den(a_d))$ is a divisor of 
$\rho_0=\den(c)$. 
Thus $\chi_W(t)$ depends only on $\GCD(\rho_0, t)$. 
\end{proof}

\section{Translated lattice point enumerator} 

\label{sec:tlpe}

To verify the symmetricity or the $\GCD$-property for 
quasi-polynomials, it is necessary to compare different 
constituents of a quasi-polynomial. 
Hence, we introduce a new function for this purpose. 

\begin{definition}
Let $P\subset \R^d$ be a polytope and $c\in\R^d$. 
We define the function $L_{(P, c)}(t)=\#((c+tP)\cap \Z^d)$ 
for $t\in\Z_{>0}$,  
which we shall call the \emph{translated lattice point enumerator}.
\end{definition}

\begin{theorem}
\label{tlpe}
\begin{itemize}
\item[$(1)$] 
If $P\subset \R^d$ is a lattice polytope of dimension $d$ and 
$c\in \R^d$, then $L_{(P, c)}(t)\in\Q[t]$. 
Furthermore, $\deg L_{(P c)}(t)=d$ and 
the leading coefficient of $L_{(P, c)}(t)$ is $\relvol(P)$. 
\item[$(2)$] 
Let $P\subset \R^d$ be a lattice polytope (not necessarily 
of dimension $d$) and $c\in \R^d$. Then $L_{(P, c)}(t)\in\Q[t]$. 
Furthermore, if $L_{(P, c)}(t)\neq 0$, then it is polynomial of degree $\dim P$
with the leading coefficient $\relvol(P)$. 
\end{itemize}
\end{theorem}

\begin{proof}
We begin by proving that $(1)$ implies $(2)$. 
Suppose that $\dim P=k<d$. If $(c+\aff(P))\cap\Z^d=\varnothing$, 
then clearly 
$L_{(P,c)}(t)=0$. 
Now, assume that $(c+\aff(P))\cap\Z^d\neq\varnothing$. 
Let $c'\in (c+\aff(P))\cap\Z^d$. Then we can express it as 
\[P+c=(P+c')+(c-c'),\] 
where $P+c'$ is a lattice polytope in $(c+\aff(P))\cap\Z^d$ 
and $(c-c')\in \aff_0(P)$ serves as a translation vector. 
Since $c'+P$ is of full dimension in $\aff(c'+P)$, 
which is isomorphic to $\R^k$, we can apply $(1)$ to conclude $(2)$. 

Now, let us prove $(1)$. 
We define two sets as follows. 
\[\begin{aligned} L(t)&=\bigl((tP+[0,c])\backslash (c+tP)\bigr)\cap \Z^d \quad \text{``lost points''} \\ 
N(t)&=\bigl((tP+[0,c])\backslash tP\bigr)\cap \Z^d \quad \text{``new points''}. \\
\end{aligned} \]
Denote by $\ell (t)$ the number of lost points, 
that is, $\ell (t)=\# L(t)$, and by $n (t)$ the number of new 
points, that is, 
$n (t)=\#N(t)$. 
It is easily seen that 
\begin{equation}
\label{eq:lpc}
\#((tP+[0, c])\cap\Z^d)=L_{(P,c)}(t)+\ell (t)=L_P(t)+ n (t).
\end{equation}
To prove $L_{(P, c)}(t)$ is a polynomial, 
it is sufficient to show that $\ell (t)$ and $n (t)$ 
are polynomials. We shall prove this fact by induction on $d$. 

For $d=0$ there exists just one polytope $P=\{0\}$ 
with one translation vector $c=0$. 
Hence, $L_{(P,c)}(t)=L_P(t)=1$. 

Next, let $P\subset \R^d$ be a lattice polytope of dimension $d$ 
with $c\in \R^d$. If $c=0$, then $L_{(P,c)}(t)=L_P(t)$, 
which is equal to the Ehrhart polynomial. 

Suppose $c\neq0$. We call a face $F$ of $P$ an ``upper face'' 
if, $(tc+F)\cap P=\varnothing$ for any $t\in\R_{>0}$. 
Similarly, a face $F$ of $P$ is a ``lower face'' 
if, $(tc+F)\cap P=\varnothing$ for any $t\in\R_{<0}$. 
Then we have 
\[
N(t)=\bigcup_{\substack{F: \text{ upper face of } P,\\ 0<s\leq 1}}
\left((sc+tF)\cap\Z^d\right). 
\]
When $\bigl(c+\aff (F)\bigr)\cap \Z^d\neq\varnothing$, 
choose $c'\in\bigl(c+\aff (F)\bigr)\cap \Z^d$. 
Then $F+c=F+c'+(c-c')$ with $c-c'\in\aff_0(F)$. 
By inductive assumption, we have that 
$L_{(F+c', c-c')}(t)$ is a polynomial in $t\in\Z_{>0}$. 
By using the inclusion-exclusion principle, 
$n(t)$ is expressed as a signed sum of translated lattice point 
enumerators $L_{(F+c', c-c')}(t)$. 
Therefore, $n(t)$ is a polynomial in $t\in\Z_{>0}$. 

By definition of the volume of $P$, we have 
\[
\lim_{t\to\infty}\frac{L_{(P, c)}(t)}{t^d}=\relvol(P). 
\]
Therefore, $\deg L_{(P, c)}(t)=d$ and the leading coefficient is 
$\relvol(P)$. 
\end{proof}

\begin{example}
Let $P=\conv\{(1,0)^t,(0,1)^t,(0,2)^t,(1,3)^t,(2,1)^t\}\subset \R^2$ 
be a lattice polytope and $c=(\frac{3}{4},\frac{3}{4})^t\in \Q^d$ 
be a rational translation vector (Figure \ref{fig:exUL}). 
Then, the upper faces of $P$ 
with respect to $c$ are $[(1,3)^t,(2,1)^t],\{(1,3)^t\},\{(2,1)^t\}$, 
while the lower faces can be described as 
$[(0,2)^t,(0,1)^t],[(0,1)^t,(1,0)^t]$ with their corresponding 
vertices, which is illustrated in Figure \ref{fig:exUL}. 
The set of lost points is 
$L(1)=\bigl((P+[0,c])\backslash (c+P)\bigr)\cap \Z^2=
\{(1,0)^t,(0,1)^t,(0,2)^t,(1,1)^t\}$ and 
the newly obtained points 
$N(1)=\bigl((P+[0,c])\backslash P\bigr)\cap \Z^2=\{(2,2)^t,(2,3)^t\}$, 
as seen in Figure \ref{fig:exNL}. 
In that manner we obtain the following number of lattice points. 

\begin{center}
\begin{tabular}{|l|c|c|c|} 
\hline 
$t$ & 0 & 1 & 2 \\ 
\hline 
$L_{(P,c)}(t)$ & 0 & 5 & 17 \\ 
\hline 
$L_P(t)$ & 1 & 7 & 20 \\ 
\hline 
$l_{c+P}(t)$ & 1 & 4 & 7 \\ 
\hline 
$n_{c+P}(t)$ & 0 & 2 & 4 \\ 
\hline 
\end{tabular}
\end{center}
From this, we calculate the polynomials as follows. 
\[\begin{aligned}
L_{(P,c)}(t)& =\frac{7}{2}t^2+\frac{3}{2}t,   & L_P(t)&=\frac{7}{2}t^2+\frac{5}{2}t+1, \\
\ell (t)&=3t+1, & n (t)&=2t.
\end{aligned}\]
\end{example}

\begin{figure}[htbp]
\centering
\begin{subfigure}[t]{0.45\textwidth}
\centering
\begin{tikzpicture}[scale=1.2]

\filldraw[fill=gray!20!white, draw=black, very thin] 
(0,1)--(0,2)--(1,3)--(2,1)--(1,0)--cycle;

\fill[black] (0,0) circle (0.05);
\fill[black] (0,1) circle (0.05);
\fill[black] (0,2) circle (0.05);
\fill[black] (0,3) circle (0.05);
\fill[black] (0,4) circle (0.05);
\fill[black] (1,0) circle (0.05);
\fill[black] (1,1) circle (0.05);
\fill[black] (1,2) circle (0.05);
\fill[black] (1,3) circle (0.05);
\fill[black] (1,4) circle (0.05);
\fill[black] (2,0) circle (0.05);
\fill[black] (2,1) circle (0.05);
\fill[black] (2,2) circle (0.05);
\fill[black] (2,3) circle (0.05);
\fill[black] (2,4) circle (0.05);
\fill[black] (3,0) circle (0.05);
\fill[black] (3,1) circle (0.05);
\fill[black] (3,2) circle (0.05);
\fill[black] (3,3) circle (0.05);
\fill[black] (3,4) circle (0.05);
\draw[->] (-0.5,0)--(3.25,0);
\draw[->] (0,-0.5)--(0,4.25);

\draw[ultra thick] (1,3)--(2,1)node[pos=0.3,right]{upper faces};
\draw[ultra thick] (1,0)--(0,1)--(0,2) node[midway, left]{lower faces};

\draw[->,red,thick] (0,0)--(0.75,0.75) node[right]{$c$};

\end{tikzpicture}
\caption{Upper and lower facets of $P$.}
\label{fig:exUL}
\end{subfigure}
\begin{subfigure}[t]{0.45\textwidth}
\centering
\begin{tikzpicture}[scale=1.2]

\fill[black] (0,0) circle (0.05);
\fill[black] (0,1) circle (0.05);
\fill[black] (0,2) circle (0.05);
\fill[black] (0,3) circle (0.05);
\fill[black] (0,4) circle (0.05);
\fill[black] (1,0) circle (0.05);
\fill[black] (1,1) circle (0.05);
\fill[black] (1,2) circle (0.05);
\fill[black] (1,3) circle (0.05);
\fill[black] (1,4) circle (0.05);
\fill[black] (2,0) circle (0.05);
\fill[black] (2,1) circle (0.05);
\fill[black] (2,2) circle (0.05);
\fill[black] (2,3) circle (0.05);
\fill[black] (2,4) circle (0.05);
\fill[black] (3,0) circle (0.05);
\fill[black] (3,1) circle (0.05);
\fill[black] (3,2) circle (0.05);
\fill[black] (3,3) circle (0.05);
\fill[black] (3,4) circle (0.05);

\draw[->] (-0.5,0)--(3.25,0);
\draw[->] (0,-0.5)--(0,4.25);

\draw[-, green, thin] (1.33,3.33)--(2.33,1.33);
\draw[-, green, thin] (1.66,3.66)--(2.66,1.66);

\draw[-, red, thin] (1,0)--(0,1)--(0,2);
\draw[-, red, thin] (1.5,0.5)--(0.5,1.5)--(0.5,2.5);

\filldraw[fill=green, draw=black] (2,2) circle (2pt) ;
\filldraw[fill=green, draw=black] (2,3) circle (2pt) node[above right]{$N$};

\filldraw[fill=red, draw=black] (1,0) circle (2pt) ;
\filldraw[fill=red, draw=black] (0,1) circle (2pt) ;
\filldraw[fill=red, draw=black] (0,2) circle (2pt) ;
\filldraw[fill=red, draw=black] (1,1) circle (2pt) node[above right]{$L$};

\end{tikzpicture}
\caption{The sets $N$ and $L$ for the polytope $P$.}
\label{fig:exNL}
\end{subfigure}
\label{fig:newlost}
\caption{New and lost points}
\end{figure}
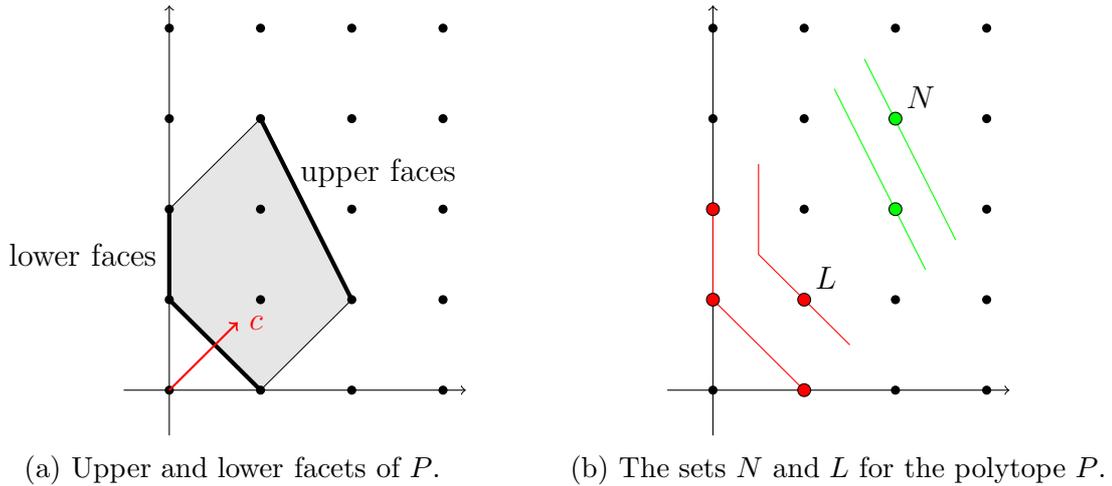

We can express constituents of the Ehrhart quasi-polynomial 
using the translated lattice point enumerator. 

\begin{corollary}
\label{cor:consti}
Let $P\in \R^d$ be a lattice polytope, and $c\in\Q^d$. 
Then the $k$-th constituent of $L_{c+P}(t)$ is $L_{(P, kc)}(t)$. 
\end{corollary}

\begin{proof}
Let $f_k$ be the $k$-th constituent of the 
Ehrhart quasi-polynomial $L_{c+P}$. 
Let $\rho=\den(c)$. Then $L_{c+P}$ has a period $\rho$. 
Since a translation by a vector in $\Z^d$ does not affect 
the number of lattice points, we get
\[\begin{aligned}
f_k(t)&=\#(t(P+c)\cap \Z^d) \quad \text{for }t\equiv k\mod\rho\\
   &=\#((tP+tc)\cap \Z^d) \quad \text{for }t\equiv k\mod\rho\\
   &=\#((tP+kc)\cap \Z^d) \quad \text{for }t\equiv k\mod\rho\\
   &=L_{(P,kc)}(t) \quad \text{for }t\equiv k\mod\rho\\
\end{aligned}\]
\end{proof}

We now prove that 
the Ehrhart quasi-polynomial of any 
almost integral centrally symmetric polytope is symmetric. 

\begin{corollary}\label{sym1}
Let $P\in \R^d$ be a centrally symmetric lattice polytope. 
Then, for every $c\in\Q^d$, the Ehrhart quasi-polynomial 
$L_{c+P}(t)$ is symmetric. 
\end{corollary}

\begin{proof}
Let $f_0,\ldots ,f_{\rho-1}$ be the constituents of $L_{c+P}(t)$. 
Note that since $P$ is centrally symmetric, $c+P$ and $-c+P$ 
contain the same number of lattice points. 
For $k\in \{1,\ldots,\rho-1\}$, by Corollary \ref{cor:consti}, 
we have 
\[\begin{aligned}
f_k(t)&=L_{(P,kc)}(t)\\
	  &=\#(kc+tP)\cap\Z^d \\
	  &=\#(-kc+tP)\cap \Z^d \\
	  &=\#((\rho-k)c+tP)\cap \Z^d\\
	  &=L_{(P,(\rho-k)c)}(t)=f_{\rho-k}(t). 
	\end{aligned}\] 
Therefore, $L_{c+P}(t)$ is symmetric. 
\end{proof}

\section{Characterizing centrally symmetric polytopes}
\label{sec:symm}

Recall that a polytope $P$ is characterized up to translation 
by the normal vectors and the $(d-1)$-volumes of its facets 
(this fact was first proved by Minkowski \cite{min}. See 
also \cite{gov, gru, sch}). 
From this fact, it follows. 

\begin{lemma}\label{facets}
Let $P\subset \R^d$ be a $d$-polytope. Then $P$ is centrally 
symmetric if and only if for each facet $F$ there exists a 
parallel facet $F^{\op}$, such that 
$\vol_{d-1}(F)=\vol_{d-1}(F^{\op})$, 
where $\vol_{d-1}$ is the $(d-1)$-dimensional Euclidean volume.
\end{lemma} 

\begin{proof}
If $P$ is centrally symmetric, then clearly 
a facet $F$ and its $F^{\op}$ have the same volume.  
Conversely, if $\vol_{d-1}(F)=\vol_{d-1}(F^{\op})$ holds for all facets, 
$P$ and $-P$ have the same data of normal vectors and $(d-1)$-volumes. 
It follows from Minkowski's result, that $P$ is centrally symmetric. 
\end{proof}

The following is the first main result of this article. 
Centrally symmetric 
polytopes are characterized by the symmetricity of 
the Ehrhart quasi-polynomials of their translations. 

\begin{theorem}\label{charsym}
Let $P\subset \R^d$ be a lattice polytope. 
Then the following are equivalent. 
\begin{itemize}
\item[$(i)$] $P$ is centrally symmetric. 
\item[$(ii)$] $L_{c+P}(t)$ is symmetric for every $c\in\Q^d$. 
\end{itemize}
\end{theorem}
\begin{proof}
The implication 
$(i) \Rightarrow (ii)$ was proved in Corollary \ref{sym1}. 
Now, let $P\subset \R^d$ be a non-centrally symmetric 
polytope of dimension $m$. 
We will prove that there exists a translation vector 
$c\in \Q^d$ such that $L_{(P,c)}(t)\neq L_{(P,-c)}(t)$ for some 
$t\in \Z_{>0}$. 
Since $P$ is not centrally symmetric, 
by using Lemma \ref{facets}, we find a facet $F$ that has 
either no parallel facet or a parallel facet $F^{\op}$ with 
different $(m-1)$-dimensional Euclidean volumes 
$\vol_{m-1}(F)\neq\vol_{m-1}(F^{\op})$. 
We only consider the latter case, because the former case 
can be considered as the latter case with $\vol_{m-1}(F^{\op})=0$. 
Since $F$ and $F^{\op}$ are parallel, the unit cubes in 
$\aff(F)\cap \Z^d$ and $\aff(F^{\op})\cap \Z^d$ have 
the same $(m-1)$-dimensional Euclidean volume. 
On the other hand $F$ and $F^{\op}$ have different volumes, 
so their relative volumes are different. 
This means for the Ehrhart polynomials 
$L_F(t)=c_{m-1}t^{m-1}+\ldots +c_0$ and 
$L_{F^{\op}}(t)=c'_{m-1}t^{m-1}+\ldots +c'_0$ that 
$c_{m-1}\neq c'_{m-1}$. 
Without loss of generality, we may assume that $c_{m-1}>c'_{m-1}$. 
Let $c\in \aff(F)_0=\aff(F^{\op})_0$ 
be a nonzero vector. By choosing $c$ generically 
(we also note that $c$ can be chosen arbitrarily small), 
we may assume $(\aff(X)+c)\cap\Z^d=\varnothing$ for every proper 
face $X\neq F, F^{\op}$. 
Then $(tX+c)\cap\Z^d=\varnothing$ for every positive 
integer $t$. 
From the argument above and Theorem \ref{tlpe}, 
the leading coefficients of 
$L_{(F, c)}(t)$ and $L_{(F^{\op}, -c)}(t)$ are different. Hence, 
we have $L_{(F, c)}(t)\neq L_{(F^{\op}, -c)}(t)$ and 
$L_{(X, c)}(t)=0$ for other proper faces $X$. Note that 
by genericity of $c$, $X$ is either an upper or a lower face. 

Next let $c'\in\aff_0(P)$ be constructed by 
inclining $c$ slightly  so that $F$ becomes 
an upper face (see Figure \ref{fig:sym}) and 
$F^{\op}$ becomes a lower face and every 
other face remains in its status 
(in particular, $c'+\partial P$ does not contain lattice points). 
Then we have 
\begin{equation}
\label{prime}
\begin{split}
L_{(P, c')}(t)&=L_{(P, c)}(t)-L_{(F^{\op}, c)}(t),\\
L_{(P, -c')}(t)&=L_{(P, -c)}(t)-L_{(F, -c)}(t).
\end{split}
\end{equation}
Next let $c''\in \Q^d$ be obtained by inclining $c$ slightly inward 
(Figure \ref{fig:sym}), such that $F$ becomes a lower face and 
$F^{\op}$ becomes an upper face. 
Then we have 
\begin{equation}
\label{second}
\begin{split}
L_{(P, c'')}(t)&=L_{(P, c)}(t)-L_{(F, c)}(t),\\
L_{(P, -c'')}(t)&=L_{(P, -c)}(t)-L_{(F^{\op}, -c)}(t).
\end{split}
\end{equation}
Now suppose that both the equations 
\begin{itemize}
\item[(a)] $L_{(P, c')}(t)=L_{(P, -c')}(t)$ and 
\item[(b)] $L_{(P, c'')}(t)=L_{(P, -c'')}(t)$  
\end{itemize}
hold. Then (\ref{prime}) and (\ref{second}) deduce 
\begin{equation}
\begin{split}
L_{(P, c)}(t)-L_{(P, -c)}(t)
&=
L_{(F^{\op}, c)}(t)-L_{(F, -c)}(t)\\
&=
L_{(F, c)}(t)-L_{(F^{\op}, -c)}(t). 
\end{split}
\end{equation}
Recall that 
the leading coefficients of $L_{(F, \pm c)}(t)$ and 
$L_{(F^{\op}, \pm c)}(t)$ 
are $c_{m-1}$ and $c'_{m-1}$, respectively. Therefore, 
the leading coefficient of 
$L_{(F^{\op}, c)}(t)-L_{(F, -c)}(t)$ is negative, while that of 
$L_{(F, c)}(t)-L_{(F^{\op}, -c)}(t)$ is positive. 
This is a contradiction. 
\end{proof}

In the proof of the above theorem, 
we may suppose (a) does not hold, 
i.e., $L_{(P, c')}(t)\neq L_{(P, -c')}(t)$. 
Then $c'$ can be perturbed slightly. That is, 
there exists a small open neighborhood 
$U\subset\aff_0(P)$ of $c'$ such that 
every replacement of $c'$ by an element of 
$U\cap\Q^d$ works similarly. 
Thus we have the following. 

\begin{corollary}
\label{openness}
Let $P\subset \R^d$ be a non-centrally symmetric lattice polytope. 
Then there exists an open set $U\subset \aff_0(P)$ such that 
$U$ intersects with any open neighborhood of $0\in\aff_0(P)$ and 
$L_{(P,c)}\neq L_{(P,-c)}$ for every translation vector $c\in U$. 
\end{corollary}

\begin{figure}[htbp]
\centering
\begin{tikzpicture}[scale=1.2]

\filldraw[fill=gray!20!white, draw=black, very thin] (1,1)--(0,2)--(2,4)--(3,4)--(5,2)--(4,1)--cycle ;

\fill[black] (0,0) circle (0.05);
\fill[black] (0,1) circle (0.05);
\fill[black] (0,2) circle (0.05);
\fill[black] (0,3) circle (0.05);
\fill[black] (0,4) circle (0.05);
\fill[black] (0,5) circle (0.05);
\fill[black] (1,0) circle (0.05);
\fill[black] (1,1) circle (0.05);
\fill[black] (1,2) circle (0.05);
\fill[black] (1,3) circle (0.05);
\fill[black] (1,4) circle (0.05);
\fill[black] (1,5) circle (0.05);
\fill[black] (2,0) circle (0.05);
\fill[black] (2,1) circle (0.05);
\fill[black] (2,2) circle (0.05);
\fill[black] (2,3) circle (0.05);
\fill[black] (2,4) circle (0.05);
\fill[black] (2,5) circle (0.05);
\fill[black] (3,0) circle (0.05);
\fill[black] (3,1) circle (0.05);
\fill[black] (3,2) circle (0.05);
\fill[black] (3,3) circle (0.05);
\fill[black] (3,4) circle (0.05);
\fill[black] (3,5) circle (0.05);
\fill[black] (4,0) circle (0.05);
\fill[black] (4,1) circle (0.05);
\fill[black] (4,2) circle (0.05);
\fill[black] (4,3) circle (0.05);
\fill[black] (4,4) circle (0.05);
\fill[black] (4,5) circle (0.05);
\fill[black] (5,0) circle (0.05);
\fill[black] (5,1) circle (0.05);
\fill[black] (5,2) circle (0.05);
\fill[black] (5,3) circle (0.05);
\fill[black] (5,4) circle (0.05);
\fill[black] (5,5) circle (0.05);

\draw (2.5,1) node[below]{$F$};
\draw (2.5,4) node[above]{$F^{\op}$};

\draw [->](4,1)--(4.5,1) node[right]{$c$};
\draw [->](4,1)--(4.5,0.8) node[below right]{$c'$};
\draw [->](4,1)--(4.5,1.2) node[above right]{$c''$};

\end{tikzpicture}
\caption{Translation by $c,c'$ and $c''$.}
\label{fig:sym}
\end{figure}
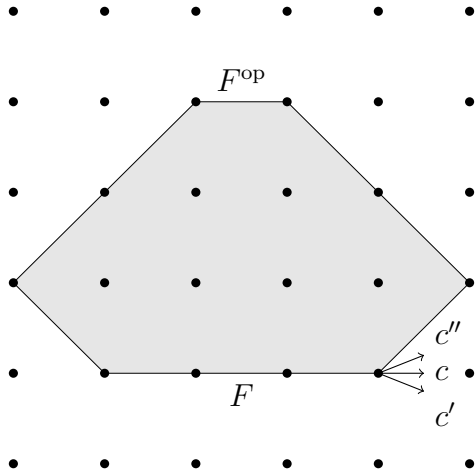

\section{Characterizing zonotopes}
\label{sec:zono}

This section deals with a characterization of lattice zonotopes in a 
similar way to Theorem \ref{charsym}. Specifically, 
if the Ehrhart quasi-polynomial of every rational shift of a 
lattice polytope satisfies the $\GCD$-property, then the 
polytope is a zonotope. 
Notice that we have already proved this statement 
for non-centrally symmetric polytopes. In fact, 
$\GCD (1,\rho)=\GCD(\rho-1,\rho)$ for all periods 
$\rho\ge 1$, but by Theorem \ref{charsym} we find a 
$c\in \Q^d$ such that $L_{(P,c)}\neq L_{(P,-c)}$. 
In order to construct a translation vector for the 
centrally symmetric polytopes that are no zonotopes, 
we consider almost constant functions:

\begin{definition}
A function $f\colon \R \to \R$ is called 
\textit{almost locally constant}, if it is locally constant 
except for a discrete set of points.
A function $f\colon \R \to \R$ is called 
\textit{almost constant}, if it is constant except for a 
discrete set of points.
\end{definition}

Let $P\subset \R^d$ be a polytope and $c\in \R^d$. 
Consider the function $L^P_c\colon \R \to \R$ defined by 
$x\mapsto \#\left((xc+P)\cap \Z^d\right)$. 
It is an almost locally constant function. 
If furthermore $c\in\Q^d$, then $L^P_c(t)$ 
is periodic with a period $\rho_0=\den(c)$. 

\begin{proposition}
\label{nongcd}
Let $P\subset\R^d$ be a lattice $d$-polytope and $c\in\Q^d$. 
If $L^P_c(x)$ is not almost constant, then there exists $c'\in\Q^d$ 
such that $L_{c'+P}(t)$ does not satisfy the $\GCD$-property. 
\end{proposition}

\begin{proof}
Since $L^P_{kc}(x/k)=L^P_c(x)$ for $k\in\Z_{>0}$, 
we may assume that $c\in\Z^d$. 
Consider the function $\delta(x):=L^P_c(x)-L^P_c(2x)$. 
Then $\delta(x)$ is 
clearly an almost locally constant function, which is $0$ near $x=0$. 
That is, there 
exists $\varepsilon>0$ such that $L^P_c(x)-L^P_c(2x)$ 
is 
constantly $0$ on the interval $(0, \varepsilon)$. 
We shall prove that 
$\delta(x)$ is not almost constant. Suppose the contrary. 
Then $\delta(x)=0$ except for a discrete set of points. 
By induction on $n>0$, we have $L^P_c(x)=L^P_c(2^nx)$ for 
almost all $x\in(0, \varepsilon)$. 
Since $L^P_c(x)$ is periodic, $L^P_c(x)$ is 
almost constant, which contradicts the assumption. 
Hence $\delta(x)$ is 
not almost constant, and there exists an interval 
$(a, b)\subset\R$ such that 
$\delta(x)\neq 0$ for all $x\in (a, b)$. Let $x\in (a, b)\cap \Q$ 
be a rational number such that $\den(xc)$ is odd. 
Then $L^P_c(x)\neq L^P_c(2x)$, 
which is equivalent to $L_{(P, xc)}(1)\neq L_{(P, 2xc)}(1)$. 

Consider the polytope $xc+P$. 
Since coordinates of $xc\in\Q^d$ have odd denominators, 
the minimal period $\rho_0$ is an odd integer. 
Therefore $\GCD(1, \rho_0)=\GCD(2, \rho_0)=1$. 
However, the argument above implies 
$L_{(P, xc)}(t)\neq L_{(P, 2xc)}(t)$. Hence the first and 
the second constituents of the Ehrhart quasi-polynomial 
$L_{xc+P}(t)$ are different, and $L_{xc+P}(t)$ does not 
have the $\GCD$-property. 
\end{proof}

\begin{example}
Let $P_1=\conv(\begin{pmatrix} 0 \\ 0 \end{pmatrix},\begin{pmatrix} 1 \\ 0 \end{pmatrix},\begin{pmatrix} 0 \\ 1 \end{pmatrix},\begin{pmatrix} 1 \\ 1 \end{pmatrix})\subset \R^2$ and $c_1=\begin{pmatrix}
\frac{1}{2} \\ \frac{1}{4}
\end{pmatrix}\in \Q^2$. 
Note that $P_1$ is a zonotope. Then,
\[L_{c_1}^{P_1}(x)=\begin{cases} 4 & \text{if } x\in 4\Z_{\ge 0}  \\ 
2 & \text{if } x\in 2+4\Z_{\ge 0}  \\
1 & \text{else} \end{cases}\] is almost constant, as illustrated in Figure \ref{fig:example1}. 

In contrast to $P_1$, consider the $3$-dimensional cross-polytope $P_2=\conv(\pm e_i|i=1,2,3)\subset \R^3$ and $c_2=\frac{1}{3}\begin{pmatrix}
1 \\ 1 \\ 1
\end{pmatrix}\in \Q^3$. 
We observe that $L_{c_2}^{P_2}$, see Figure \ref{fig:example2}, 
is not almost constant:
\[L_{c_2}^{P_2}(x)=\begin{cases} 7 & \text{if } x\in 3\Z_{\ge 0}  \\ 
1 & \text{if } x\in (k,k+1] \text{ for } k\in 3\Z_{\ge 0}  \text{ or } x\in [k-1,k) \text{ for } k\in 3\Z_{> 0}\\
0 & \text{else} \end{cases}.\] 
Therefore, we can find, for example, $x=\frac{3}{5}$ for that $L_{(xc+P_2)}$ does not satisfy the $\GCD$-property. 
More explicitly, the Ehrhart quasi-polynomial of the octahedron $P_2$ translated by $xc=(\frac{1}{5},\frac{1}{5},\frac{1}{5})^t$ is given by 
\[L_{xc+P_2}(n)=\begin{cases} 
		\frac{4}{3} n^3+2 n^2 + \frac{8}{3} n + 1 & \text{if } n \equiv 0 \mod 5 \\			
		\frac{4}{3} n^3- \frac{1}{3} n  & \text{if } n \equiv 1 \mod 5 \\ 
		\frac{4}{3} n^3- \frac{4}{3} n  & \text{if } n \equiv 2 \mod 5 \\
		\frac{4}{3} n^3-  \frac{4}{3} n  & \text{if } n \equiv 3 \mod 5 \\
		\frac{4}{3} n^3-  \frac{1}{3} n  & \text{if } n \equiv 4 \mod 5 \\ \end{cases}, \]
for which the first and second constituents are different, but $\GCD (1,5)=\GCD (2,5)$.
\end{example}

\begin{figure}[htbp]
\centering
\begin{subfigure}[t]{0.45\textwidth}
\centering
\begin{tikzpicture}[scale=1.2]

\draw[->] (-0.5,0)--(4.25,0);
\draw[->] (0,-0.5)--(0,4.25);
\filldraw[black] (0,4) circle (2pt);
\filldraw[black] (1,2) circle (2pt);
\filldraw[black] (2,4) circle (2pt);
\filldraw[black] (3,2) circle (2pt);
\filldraw[black] (4,4) circle (2pt);
\draw[line width=2.5pt] (0,1)--(4,1);
\filldraw[fill=white, draw=black] (0,1) circle (2pt) ;
\filldraw[fill=white, draw=black] (1,1) circle (2pt) ;
\filldraw[fill=white, draw=black] (2,1) circle (2pt) ;
\filldraw[fill=white, draw=black] (3,1) circle (2pt) ;
\filldraw[fill=white, draw=black] (4,1) circle (2pt) ;

\draw (0,1) node[left]{$1$};
\draw (0,2) node[left]{$2$};
\draw (0,3) node[left]{$3$};
\draw (0,4) node[left]{$4$};

\draw (1,0) node[below]{$2$};
\draw (2,0) node[below]{$4$};
\draw (3,0) node[below]{$6$};
\draw (4,0) node[below]{$8$};

\end{tikzpicture}
\caption{$L^{P_1}_{c_1}$}
\label{fig:example1}
\end{subfigure}
\begin{subfigure}[t]{0.45\textwidth}
\centering
\begin{tikzpicture}[scale=1.2]

\filldraw[black] (0,4) circle (2pt);
\filldraw[black] (3,4) circle (2pt);

\draw[->] (-0.5,0)--(4.25,0);
\draw[->] (0,-0.5)--(0,4.25);

\draw [line width=2.5pt] (0,0.57)--(1,0.57);
\draw [line width=2.5pt] (1,0)--(2,0);
\draw [line width=2.5pt] (2,0.57)--(4,0.57);
\draw [line width=2.5pt] (4,0)--(4.2,0);

\filldraw[fill=white, draw=black] (0,0.57) circle (2pt) ;
\filldraw[fill=white, draw=black] (3,0.57) circle (2pt) ;

\draw (0,0.57) node[left]{$1$};
\draw (0,1.14) node[left]{$2$};
\draw (0,1.71) node[left]{$3$};
\draw (0,2.28) node[left]{$4$};
\draw (0,2.85) node[left]{$5$};
\draw (0,3.42) node[left]{$6$};
\draw (0,4) node[left]{$7$};

\draw (1,0) node[below]{$1$};
\draw (2,0) node[below]{$2$};
\draw (3,0) node[below]{$3$};
\draw (4,0) node[below]{$4$};

\filldraw[black] (1,0.57) circle (2pt);
\filldraw[black] (2,0.57) circle (2pt);
\filldraw[black] (4,0.57) circle (2pt);

\filldraw[fill=white, draw=black] (1,0) circle (2pt) ;
\filldraw[fill=white, draw=black] (2,0) circle (2pt) ;
\filldraw[fill=white, draw=black] (4,0) circle (2pt) ;

\end{tikzpicture}
\caption{$L^{P_2}_{c_2}$}
\label{fig:example2}
\end{subfigure}
\caption{Examples almost locally constant functions}
\end{figure}
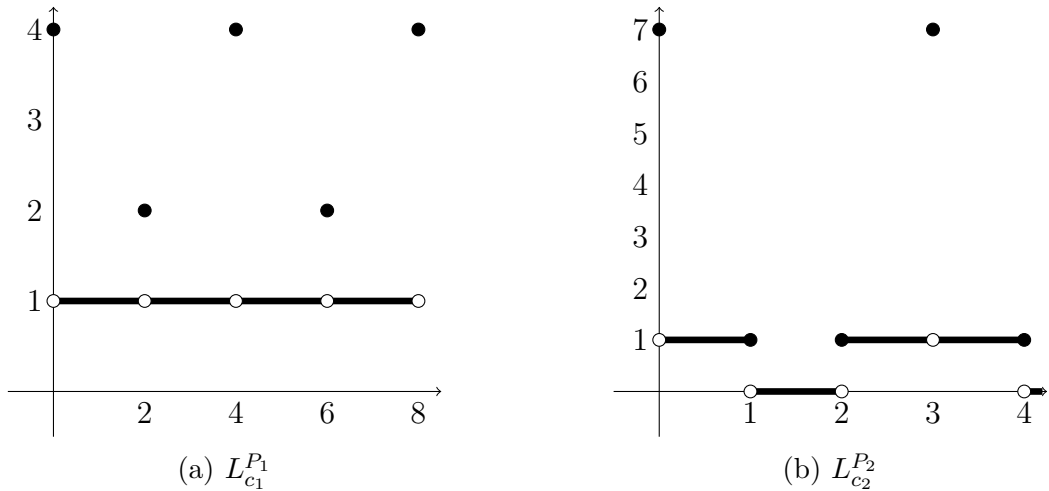

We will use the following characterization of a zonotope. 
\begin{proposition}\cite{mc}
\label{gov}
Let $P\subset R^d$ be a polytope of dimension $m\le d$. 
Then, $P$ is a zonotope if and only if 
all faces of dimension $j$ are centrally symmetric for some 
$2\le j \le m-2$.
\end{proposition}
The second main result of this paper is the following. 
\begin{theorem}
\label{charzono}
Let $P\subset\R^d$ be a lattice polytope. Then the following are equivalent
\begin{itemize}
\item[$(i)$] $P$ is a zonotope. 
\item[$(ii)$] $L_{c+P}(t)$ satisfies the $\GCD$-property for every $c\in\Q^d$. 
\end{itemize}
\end{theorem}

\begin{proof}
The implication 
$(i) \Rightarrow (ii)$ was proved in Theorem \ref{zono1}. 
Now, let $P\subset \R^d$ be an $m$-polytope that is not a 
zonotope. We aim to prove the existence of $c\in\Q^d$ 
such that the Ehrhart quasi-polynomial $L_{c+P}(t)$ does not 
satisfy the $\GCD$-property. 

We can categorize polytopes that are not zonotopes into 
three types: 
\begin{itemize}
\item[$(a)$] 
non-centrally symmetric polytopes, 
\item[$(b)$] 
centrally symmetric polytopes with at least one non-centrally 
symmetric facet, and 
\item[$(c)$] 
centrally symmetric polytopes whose facets are all rentrally symmetric. 
\end{itemize}

First, let $P$ be a non-centrally symmetric polytope. 
As mentioned earlier, this case has been addressed in 
Theorem \ref{charsym}.

Second, suppose $P$ is a centrally symmetric polytope 
with at least one facet $F$ that is not centrally symmetric. 
Let $F\subset P$ be a non-symmetric facet with its opposite 
facet $F^{\op}$. 
By Corollary \ref{openness}, there exists an open subset 
$U\subset\aff_0(F)$ 
such that for every $c'\in U\cap\Q^d$, 
there exists $t\in\Z_{>0}$ such that 
$L_{(F, c')}(t)\neq L_{(F, -c')}(t)(=L_{(F^{\op}, c')}(t))$. 
Now choose $c'\in U\cap\Q^d$ with odd denominators and 
generic enough so that we may assume that 
$(c'+\aff(X))\cap\Z^d=\varnothing$ for all faces 
$X\neq F, F^{\op}$ of $P$. 
Let $c''\in\Z^d$ be an integral vector such that 
$F$ is an upper face and 
$F^{\op}$ is a lower face. Consider $c=c'+c''$. 
Then, $L^{tP}_{c}(x)$ is not almost constant. Because, at $x=1$, 
the number of lattice points in the upper faces $\#((c+tF)\cap\Z^d)$ 
and that of the lower faces $\#((c+tF^{\op})\cap\Z^d)$ are different. 
Hence the function $L^{tP}_{c}(x)$ takes different values on 
the interval $x\in (1-\varepsilon, 1)$ and $x\in(1, 1+\varepsilon)$. 
Thus, by Proposition \ref{nongcd}, 
the Ehrhart quasi-polynomial $L_{c+P}(t)$ does not satisfy 
the $\GCD$-property. 

Finally, consider a centrally symmetric polytope $P$ whose 
facets are also centrally symmetric, but that is not a zonotope. 
By using Proposition \ref{gov}, there exists a face 
$G$ of dimension $(m-2)$ that is not centrally symmetric. 
Let $F_0$ and $F_1$ be facets of $P$ such that 
$G_0:=G=F_0\cap F_1$. Let $G_1$ be the opposite 
facet of $G_0$ in $F_1$. 
Since $F_1$ is centrally symmetric, $G_1$ is a translation of 
$G^{\op}$ which itself is a translate of $-G$. In particular, 
$G_1$ is parallel to $G_0$. 
Continue like this and let $F_i$ be the facet of $P$ 
such that $G_{i-1}=F_{i-1}\cap F_i$ and let $G_i$ be 
the opposite face of $G_{i-1}$ in $F_i$. 
Observe that by this construction, $G_i$ is a translation of $G$ 
(resp. $G^{\op}$), if $i$ is even (resp. odd). This process terminates 
at some $n\in \N$. Then $n$ is even. 
Because, if $n$ is odd, $G$ would be a translation of $-G$, 
which contradicts the assumption. 
Furthermore, since $P$ is centrally symmetric, 
$G_\frac{n}{2}$ is opposite of $G$ and thus, 
a translation of $-G$. Hence, $\frac{n}{2}$ is odd.

Let $\pi:\R^d\to \R^2$ be the orthogonal projection to 
the orthogonal complement $\aff(G)^{\perp}$ of $\aff(G)$. 
Then, the image $\pi(P)$ is an $n$-gon with vertices 
$v_1,\ldots ,v_n (=v_0)$ and edges $u_1,\ldots , u_n$. 
We choose the numbering so that 
$G_i=\pi^{-1}(v_i)\cap P$ and $F_i=\pi^{-1}(u_i)\cap P$. 

We claim that $\aff_0(G)$ is transversal to $\aff_0(F)$ 
for all facets $F\neq F_i$, $i=1,\ldots ,n$. 
Since $F$ is of codimension $1$, 
we have $\aff_0(G)+\aff_0(F)$ is either $\aff_0(F)$ or 
$\aff_0(G)+\aff_0(F)=\aff_0(P)$. 
The latter case means that $\aff_0(G)$ and $\aff_0(F)$ 
are transversal. The former case is equivalent to 
$\aff_0(G)\subseteq \aff_0(F)$. Then the image 
$\pi(F)\subset \pi(P)$ is $1$-dimensional. 
However, since $F$ does not separate $P$, 
$\pi(F)$ does not separate $\pi(P)$. Thus, $\pi(F)=u_i$ 
for some $i\in \{1,\ldots, n\}$ and $F=F_i$.

By Corollary \ref{openness}, 
we are able to choose a sufficiently small 
rational vector $c\in \aff_0(G)\cap\Q^d$ 
such that $\#((c+G)\cap \Z^d)\neq \#((-c+G)\cap \Z^d)$ and 
$(c+(\partial P \backslash (F_1\cup F_2 \cup \ldots \cup F_n)))
\cap \Z^d=\varnothing$. The latter is possible since $\aff_0(G)$ 
is transversal to $\aff_0(F)$ for all facets $F\neq F_i$, $i=1,\ldots ,n$. 

\begin{figure}[htbp]
\centering
\begin{tikzpicture}[scale=1.2]

\draw (-0.7,0) node[below]{$v_0$}
	--(0.7,0) node[below]{$v_1$}
	--(1.8,0.8) node[below right]{$v_2$}
	--(2.3,2) node[right]{$v_3$}
	--(1.8,3.2) node[above right]{$v_4$}
	--(0.7,4) node[above]{$v_5$}
	--(-0.7,4) node[above]{$v_6$}
	--(-1.8,3.2) node[above left]{$v_7$}
	--(-2.3,2) node[left]{$v_8$}
	--(-1.8,0.8) node[below left]{$v_9$} --cycle;

\draw [->](0.7,0)--(1.7,0) node[right]{$v$};
\end{tikzpicture}
\caption{projection ($n=10$)}
\label{fig:example}
\end{figure}
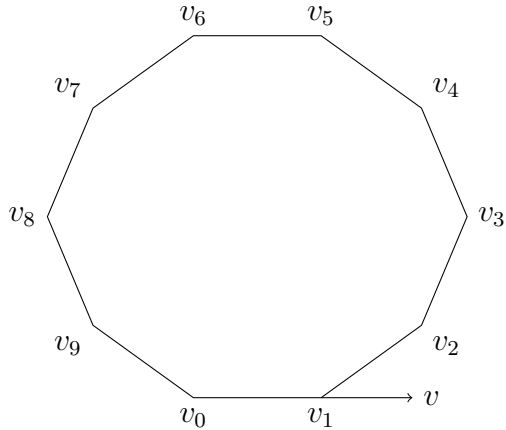

Lastly, take a vector 
$v\in \aff_0(F_1)\cap\Z^d=\aff_0(F_{\frac{n}{2}+1})\cap\Z^d$ 
such that $F_2,\ldots, F_{\frac{n}{2}}$ are upper faces, and 
$F_{\frac{n}{2}+2},\ldots, F_{n}$ are lower faces with respect to $v$. 
Now, we count the number of lattice points lost and newly obtained 
by translating $P$ by $c'=c+v$. 
The opposite face of $F_i$ is $F_{i+\frac{n}{2}}$, 
where the indices are considered modulo $n$. 
In addition, $F_{i+\frac{n}{2}}$ is a translation of $-F$ 
which itself is a translation of $F$, since all facets are symmetric. 
In order to simplify notation, let $(X)_{\Z}$ denote 
$X\cap \Z^d$ for a set $X\subset \R^d$. 
Since $\frac{n}{2}$ is odd, the number of lattice points on 
the upper faces of $P$ with respect to $c'$ is 

\[\begin{aligned} & \interior (c+F_2)_{\Z}+\ldots +\interior (c+F_{\frac{n}{2}})_{\Z}+(c+G_1)_{\Z}+\ldots + (c+G_{\frac{n}{2}})\\
&= \sum_{i=2}^{\frac{n}{2}}  \interior (c+F_i)_{\Z}+ \frac{n+2}{4}(-c+G)_{\Z}+ \frac{n-2}{4}(c+G)_{\Z}.\end{aligned}\]

Whereas the number of lattice points on the lower faces is 

\[\begin{aligned} & \interior (c+F_{\frac{n}{2}+2})_{\Z}+\ldots +\interior (c+F_n)_{\Z}+(c+G_{\frac{n}{2}+1})_{\Z}+\ldots + (c+G_n)\\
&= \sum_{i=2}^{\frac{n}{2}}  \interior (c+F_{i+\frac{n}{2}})_{\Z}+ \frac{n-2}{4}(-c+G)_{\Z}+ \frac{n+2}{4}(c+G)_{\Z}.\end{aligned}\]

From $\#(c+G)\cap \Z^d\neq \#(-c+G)\cap \Z^d$, 
it follows that these numbers are different. 
Therefore, we have 
$\#\left((xc'+P)\cap \Z^d\right)\neq \#\left((x'c'+P)\cap \Z^d\right)$ for $x\in (1-\epsilon,1), x'\in (1, 1+\epsilon)$ with $\epsilon \in \R$ small enough. Thus, $L_{c'}^P$ is not almost constant. By 
Proposition \ref{nongcd}, the Ehrhart quasi-polynomial 
$L_{c'+P}(t)$ does not satisfy the $\GCD$-property. 
\end{proof}

\section{Discussions and further problems}

\label{discuss}

\subsection{Minimal periods}

\label{sec:min}

In Ehrhart theory, the minimal period of an Ehrhart quasi-polynomial 
sometimes becomes strictly smaller than the LCM of denominators 
of vertices. This is known as the 
\emph{period collapse} phenomenon. 
For almost integral polytopes, it is natural to ask the following. 

\begin{problem}
\label{minpbm}
Let $P$ be a lattice polytope in $\R^d$ and $c\in\Q^d$. Is the 
minimal period of $L_{c+P}(t)$ equal to $\den(c)$? 
\end{problem}

As a partial result, we have the following for zonotopes. 

\begin{proposition}
\label{zonmax}
Let $P\subset\R^d$ be a lattice zonotope and $c\in\Q^d$. 
\begin{itemize}
\item[(1)] 
The minimal period of $L_{c+P}(t)$ is $\den(c)$. 
\item[(2)] 
The inequality 
\begin{equation}
\label{eq:ineq}
\#((c+P)\cap\Z^d)\leq \#(P\cap\Z^d) 
\end{equation}
holds. Furthermore, in (\ref{eq:ineq}), the equality holds if and only if $c\in\Z^d$. 
\end{itemize}
\end{proposition}

\begin{proof}
Consider the term $W=\varnothing$ in 
Ardila-Beck-McWhirtner's formula (Proposition \ref{aiz}). Then 
$\chi_W(t)=1$ if and only if $tc\in\Z^d$, which is equivalent to $t$ is 
divisible by $\den(c)$. This yields (1). 

(2) is the special case $t=1$. 
\end{proof}
One of the difficulties of Problem \ref{minpbm} is that Proposition \ref{zonmax} (2) 
does not hold for general polytopes. The next example shows that 
$c+P$ can contain more lattice points than the original lattice polytope $P$. 
It seems of interest in its own right. 

\begin{example}
\label{counterex}
Let $n>7$. Define the lattice polytope $P_n$ in $\R^3$ by 
\begin{equation}
\conv\{ 0, ne_1, ne_2, n(e_1+e_2), e_3, ne_2+e_3, (1-n)e_3\}, 
\end{equation}
where $e_1, e_2, e_3$ are the standard basis of $\R^3$. 
For $0<k<n$, let $c_k=\frac{k}{n}e_3$. 
Then straightforward computation shows 
\begin{equation}
|P_n\cap\Z^3|=\frac{2n^3+3n^2+19n+12}{6}, 
\end{equation}
and we have 
\begin{equation}
\alpha(n, k):=|(c_k+P_n)\cap\Z^3|-|P_n\cap\Z^3|=
k(n+1)-k^2-2n-1, 
\end{equation}
which becomes positive for some $k$, e.g., $k=3$. If $n$ is odd, $\alpha(n, k)$ attains 
maximum at $k=\frac{n+1}{2}$ and 
\begin{equation}
\label{eq:upperbd}
\alpha(n, \frac{n+1}{2}):=\frac{n^2-6n-3}{4}. 
\end{equation}
\end{example}

\begin{problem}
For which lattice $d$-polytope $P$ and $c\in\Q^d$ does the inequality 
\begin{equation}
|P\cap\Z^d|<|(c+P)\cap\Z^d|
\end{equation}
hold? For a lattice polytope $P$, what is $\max\{|(c+P)\cap\Z^d|  
c\in\Q^d\}$? 
\end{problem}

\begin{problem}
Do there exist certain constants $\kappa, \lambda>0$ 
such that the following holds for every $P$ and $c$? 
\begin{equation}
|(c+P)\cap\Z^d|\leq |P\cap\Z^d|+\kappa |P\cap\Z^d|^\lambda
\end{equation}
(The formula (\ref{eq:upperbd}) in 
Example \ref{counterex} may suggest that $\lambda\geq\frac{d-1}{d}$.) 
\end{problem}


Let us fix a lattice polytope $P$ in $\R^d$. Then the translation $P+c$ ($c\in\Q^d$) 
defines infinitely many rational polytopes. It seems natural to ask what kind of 
polynomials appear as constituents of the Ehrhart quasi-polynomial 
$L_{P+c}(t)$ of a translation $P+c$ $(c\in\Q^d)$. We pose the following. 

\begin{problem}
\label{Stanley}
Let $P\subset \R^d$ be a lattice polytope. 
We can define a map (see Theorem \ref{tlpe}) 
\[
\begin{aligned}
\Gamma_P: 
(\R/\Z)^d &\longrightarrow  \Q[t]\\
 c &\longmapsto L_{(P, c)}(t). 
\end{aligned}
\] 
What can we say about the map? For example, is the image a finite set? 
(Note that the finiteness of the image implies that the set of all constituents 
of the Ehrhart quasi-polynomials of all possible translations $P+c$ is finite.) 
\end{problem}

\subsection{Rational polytopes with the $\GCD$-property and hyperplane arrangements}

\label{sec:hyp}

In our main results, 
Theorem \ref{charsym} and \ref{charzono}, 
we assumed that $P$ is an almost integral polytope. 
The authors are uncertain whether these assumptions are 
necessary or not. Thus we pose the following question. 

\begin{problem}
Let $P\subset\R^d$ be a rational polytope. 
\begin{itemize}
\item[(1)] 
If $L_{c+P}(t)$ is symmetric for all $c\in\Q^d$, 
is $P$ a centrally symmetric (almost integral) polytope? 
\item[(2)] 
If $L_{c+P}(t)$ satisfies the $\GCD$-property for all $c\in\Q^d$, 
is $P$ a (almost integral) zonotope? 
\end{itemize}
\end{problem}

It is worth noting that 
almost integral zonotopes are not the only examples of polytopes 
whose Ehrhart quasi-polynomials satisfy the $\GCD$-property. 
For instance, some interesting rational polytopes (simplices) 
have this property. 
For example, Suter \cite{sut} observed that 
the fundamental alcove (a certain rational simplex) 
of a root system has an Ehrhart quasi-polynomial 
with the $\GCD$-property. 
The following are examples corresponding to 
the root systems of type $E_6, E_7, E_8, F_4$ and $G_2$. 

\begin{example}
\label{alcove}
Let $e_1, \dots, e_\ell$ be the standard basis of $\R^\ell$. 
\begin{itemize}
\item[(1)] 
Consider the $6$-dimensional rational simplex 
$P_{E_6}\subset\R^6$ defined by 
\begin{equation}
P_{E_6}=\conv\left\{0, e_1, e_2, \frac{1}{2}e_3, \frac{1}{2}e_4, \frac{1}{2}e_5, \frac{1}{3}e_6\right\}. 
\end{equation}
The Ehrhart quasi-polynomial $L_{P_{E_6}}(t)$ has 
the minimal period $\rho=6$ 
and satisfies the $\GCD$-property. 
\item[(2)] 
Consider the $7$-dimensional rational simplex 
$P_{E_7}\subset\R^7$ defined by 
\begin{equation}
P_{E_7}=\conv\left\{0, e_1, \frac{1}{2}e_2, \frac{1}{2}e_3, \frac{1}{2}e_4, \frac{1}{3}e_5, \frac{1}{3}e_6, \frac{1}{4}e_7\right\}. 
\end{equation}
The Ehrhart quasi-polynomial $L_{P_{E_7}}(t)$ has 
the minimal period $\rho=12$ 
and satisfies the $\GCD$-property. 
\item[(3)] 
Consider the $8$-dimensional rational simplex 
$P_{E_8}\subset\R^8$ defined by 
\begin{equation}
P_{E_8}=\conv\left\{0, \frac{1}{2}e_1, \frac{1}{2}e_2, \frac{1}{3}e_3, \frac{1}{3}e_4, \frac{1}{4}e_5, \frac{1}{4}e_6, \frac{1}{5}e_7, \frac{1}{6}e_8\right\}. 
\end{equation}
The Ehrhart quasi-polynomial $L_{P_{E_8}}(t)$ has 
the minimal period $\rho=60$ 
and satisfies the $\GCD$-property. 
\item[(4)]
Consider the $4$-dimensional rational simplex 
$P_{F_4}\subset\R^4$ defined by 
\begin{equation}
P_{F_4}=\conv\left\{0, \frac{1}{2}e_1, \frac{1}{2}e_2, \frac{1}{3}e_3, \frac{1}{4}e_4\right\}. 
\end{equation}
The Ehrhart quasi-polynomial $L_{P_{F_4}}(t)$ has 
the minimal period $\rho=12$ 
and satisfies the $\GCD$-property. 
\item[(5)]
Consider the $2$-dimensional rational simplex 
$P_{G_2}\subset\R^2$ defined by 
\begin{equation}
P_{G_2}=\conv\left\{0, \frac{1}{2}e_1, \frac{1}{3}e_2\right\}. 
\end{equation}
The Ehrhart quasi-polynomial $L_{P_{G_2}}(t)$ has 
the minimal period $\rho=6$ 
and satisfies the $\GCD$-property. 
\end{itemize}
\end{example}
As we observed in Example \ref{ex:01}, 
not all rational polytopes possess the $\GCD$-property. 
It appears to be a relatively rare phenomenon. 
This naturally leads to the following question. 

\begin{problem}
Which rational polytopes exhibit 
Ehrhart quasi-polynomials with the $\GCD$-property 
(or symmetric quasi-polynomial)? 
\end{problem}

The notion of the $\GCD$-property for quasi-polynomials 
was first formulated in the theory of hyperplane arrangements. 
Let us recall briefly. 

Consider non-zero group homomorphisms 
$\alpha_i:\Z^\ell\longrightarrow\Z$ ($i=1, \dots, n$). 
These homomorphism lead to the definition of a hyperplane 
arrangement as the set of hyperplanes defined by 
$\Ker(\alpha_i\otimes\R: \R^\ell\longrightarrow\R)$, $i=1, \dots, n$. 

On the other hand, for any positive integer $q>0$, 
$\alpha_i$ induces a homomorphism 
$\alpha_i\otimes(\Z/q\Z):(\Z/q\Z)^\ell\longrightarrow\Z/q\Z$. 
Kamiya, Takemura and Terao \cite{ktt-cent} established that 
the number of points in the ``mod $q$ complement'' 
\begin{equation}
\left|
(\Z/q\Z)^\ell\smallsetminus\bigcup_{i=1}^n \Ker(\alpha_i\otimes(\Z/q\Z))
\right|
\end{equation}
is a quasi-polynomial in $q$ with the $\GCD$-property, 
which is called 
the \emph{characteristic quasi-polynomial} of the arrangement. 

One of the most important properties is that the 
first constituent of the characteristic quasi-polynomial is equal to the 
characteristic polynomial of a hyperplane arrangement \cite{ot}. 
Moreover, the characteristic quasi-polynomial plays a significant 
role in the context of toric arrangements \cite{ers, lty, ty}. 

It is worth noting that Suter's computations 
(Example \ref{alcove}) is closely 
related to characteristic quasi-polynomials. 
In fact, the Ehrhart quasi-polynomials in Example \ref{alcove}, 
up to scalar multiples, 
are identical to characteristic quasi-polynomials of 
corresponding reflection arrangements \cite{ath, yos-wor}. 

Since the relationship between zonotopes and hyperplane arrangements 
is an actively studied research topic (e.g., see \cite[Chap 7]{zie}), 
it would be interesting to explore the following question. 

\begin{problem}
Are there direct connections between 
characteristic quasi-polynomials of hyperplane arrangements 
and Ehrhart quasi-polynomials of almost integral zonotopes? 
(Note that both are quasi-polynomials with the $\GCD$-property.) 
\end{problem}

\medskip

\noindent
{\bf Acknowledgements.} 
Christopher de Vries was supported by the German Academic Exchange Service. 
Masahiko Yoshinaga 
was partially supported by JSPS KAKENHI 
Grant Numbers JP19K21826, JP18H01115, JP23H00081. 
The authors are grateful to 
Office for International Academic Support, 
Faculty of Science, Hokkaido University for their support 
during the first author's visit to the second author. 
The authors thank 
Akihiro Higashitani, Shigetaro Tamura, Tan Nhat Tran, 
Matthias Beck, Christos Athanasiadis, Satoshi Murai, Mitsuhiro Miyazaki 
and Richard Stanley for fruitful discussions and comments
during the preparation of this paper. 
In particular, Problem \ref{Stanley} was inspired by Stanley's comment.

\end{document}